\documentclass[11pt,reqno]{amsart}
\usepackage{mathrsfs}
\usepackage{}

\def\subjclass#1{{\renewcommand{\thefootnote}{}%
\footnote{\emph{Mathematics Subject Classification (2010):} #1}}}

\usepackage{amsfonts}
\usepackage{amssymb}

\DeclareMathOperator{\curl}{curl}
\DeclareMathOperator{\divg}{div}

\hoffset=- 2cm \voffset=- 0cm 
 \theoremstyle{plain}
\newtheorem{Thm}{Theorem}

\newtheorem{Rem}[Thm]{Remark}
\newtheorem{Lem}[Thm]{Lemma}

\usepackage{amssymb}
\usepackage{color}

\newcommand {\p}{\partial}
\newcommand{\q}{\quad}

\def\0{\mathbf 0}

\def\O{\Omega}

\def\u{\mathbf u}
\def\w{\mathbf w}

\errorcontextlines=0 \numberwithin{equation}{section}
\numberwithin{Thm}{section}
\allowdisplaybreaks

\begin{document}
\large

\title[Comparison results for eigenvalues]{Comparison results for eigenvalues of curlcurl operator and Stokes operator}

\author[]{Zhibing Zhang}

\address{Zhibing Zhang: School of Mathematics and Physics, Anhui University of Technology, Ma'anshan 243032, PR China; }
\email{zhibingzhang29@126.com}

\thanks{}

\keywords{Eigenvalue problem, curlcurl operator, Stokes operator, Maxwell constants}

\subjclass{35P15; 35P05}

\begin{abstract}
This paper mainly establishes comparison results for eigenvalues of $\curl\curl$ operator and Stokes operator.
For three-dimensional simply connected bounded domains, the $k$-th eigenvalue of $\curl\curl$ operator under tangent boundary condition or normal boundary condition is strictly smaller than the $k$-th eigenvalue of Stokes operator. For any dimension $n\geq2$, the first eigenvalue of Stokes operator is strictly larger than the first eigenvalue of Dirichlet Laplacian. For three-dimensional strictly convex domains, the first eigenvalue of $\curl\curl$ operator under tangent boundary condition or normal boundary condition is strictly larger than the second eigenvalue of Neumann Laplacian.

\end{abstract}
\maketitle

\section{Introduction}
The $\curl\curl$ eigenvalue problems are motivated from the investigation of the eigenvalue problem of the Maxwell operator or the time-harmonic Maxwell equations, see \cite{CD1999,Filonov2003,HL1996}. The eigenvalues of the $\curl\curl$ operator and the Stokes operator are important because of their many applications in electromagnetic fields and fluid mechanics, respectively. Under zero tangential boundary condition or zero normal boundary condition, $0$ is the smallest and trivial eigenvalue of the $\curl\curl$ operator. We are concerned with the positive eigenvalues of the $\curl\curl$ operator. For $\lambda>0$, if $\u$ is a solution to
$$\curl\curl\u=\lambda\u,$$
then $\u$ must satisfy the compatible condition $\divg\u=0$.
In this paper, we consider the following $\curl\curl$ eigenvalue problems
\begin{equation}\label{eig-a}
\begin {cases}
\curl\curl\u=\alpha\u,\q\divg\mathbf{u}=0 & \text{ \rm in } \Omega,\\
\mathbf{u}\times\nu=\mathbf{0} &\text{ \rm on } \partial\Omega,\\
\u\in\Bbb H_2(\O)^\perp,
\end {cases}
\end{equation}
\begin{equation}\label{eig-b}
\begin {cases}
\curl\curl\u=\beta\u,\q\divg\mathbf{u}=0& \text{ \rm in } \Omega,\\
\mathbf{u}\cdot\nu=0,\q \curl\u\times\nu=\0 &\text{ \rm on } \partial\Omega,\\
\u\in\Bbb H_1(\O)^\perp,
\end {cases}
\end{equation}
and the Stokes eigenvalue problem
\begin{equation}\label{eig-c}
\begin {cases}
-\Delta\u+\nabla p=\gamma\u,\q\divg\mathbf{u}=0 & \text{ \rm in } \Omega,\\
\mathbf{u}=\mathbf{0} &\text{ \rm on } \partial\Omega,\\
\end {cases}
\end{equation}
where $\Omega\subset\mathbb{R}^3$ is a bounded domain with $C^1$ boundary $\partial\Omega$, $\nu$ is the unit outer normal vector field on $\p\O$ and the spaces $\Bbb H_1(\O)$, $\Bbb H_2(\O)$ are defined by
$$
\aligned
&\Bbb H_1(\O)=\{\u\in L^2(\O,\Bbb R^3): \curl\u=\0,\; \divg\u=0\text{ in $\O$,}\; \u\cdot\nu=0 \text{ on $\p\O$}\},\\
&\Bbb H_2(\O)=\{\u\in L^2(\O,\Bbb R^3): \curl\u=\0,\;\divg\u=0\text{ in $\O$,}\; \u\times\nu=\0 \text{ on $\p\O$}\}.
\endaligned
$$
Here we point out that the purpose of imposing $\u\in\Bbb H_2(\O)^\perp$ in \eqref{eig-a} and $\u\in\Bbb H_1(\O)^\perp$ in \eqref{eig-b} is to guarantee that $\alpha$ and $\beta$ are positive eigenvalues.

Throughout this paper, if there is no special declaration, we always make the following assumptions on the domain $\O$:
\begin{enumerate}
\item[$(a)$] $\Omega\subset\mathbb{R}^3$ is a bounded $C^1$ domain, and $\Omega$ is locally situated on one side of $\partial\Omega$; $\partial\Omega$ has $m+1$ connected components $\Gamma_0$, $\Gamma_1$, $\cdots$, $\Gamma_m$, where $\Gamma_0$ denotes the boundary
of the infinite connected component of $\mathbb{R}^3\backslash\overline{\Omega}$.
\item[$(b)$] The domain $\Omega$ which can be multiply connected, is made simply connected by $N$ regular cuts $\Sigma_1$, $\Sigma_2$, $\cdots$, $\Sigma_N$ which are of class $C^2$; the $\Sigma_i$, $i=1,2,\cdots,N$ satisfying $\Sigma_i\cap\Sigma_j=\emptyset$ for $i\neq j$ are non-tangential to $\partial\Omega$.
\end{enumerate}
We say that $\O$ is simply connected if $N=0$, and $\O$ has no holes if $m=0$. It is well-known that $\dim \Bbb H_1(\O)=N$ and $\dim \Bbb H_2(\O)=m$. Let
$$
\aligned
&0<\alpha_1\leq\alpha_2\leq\alpha_3\leq\cdots,\\
&0<\beta_1\leq\beta_2\leq\beta_3\leq\cdots,\\
&0<\gamma_1\leq\gamma_2\leq\gamma_3\leq\cdots\\
\endaligned
$$
denote the successive eigenvalues for \eqref{eig-a}, \eqref{eig-b} and \eqref{eig-c}, respectively. Let
$$0<\lambda_1<\lambda_2\leq \lambda_3\leq\cdots$$
be the eigenvalues of the Dirichlet Laplacian and
$$0=\mu_1<\mu_2\leq \mu_3\leq\cdots$$
be the eigenvalues of the Neumann Laplacian. Here each eigenvalue is repeated according to its multiplicity. We give some notations frequently used in the context as follows:
$$
\aligned
&H^1_0(\divg0,\O)=\{\u\in H_0^1(\O,\Bbb R^3):\divg\u=0\text{ in $\O$}\},\\
&H_{n0}(\divg0,\curl,\O)=\{\u,\;\curl\u\in L^2(\O,\Bbb R^3):\divg\u=0\text{ in $\O$,}\; \u\cdot\nu=0 \text{ on $\p\O$}\},\\
&H_{t0}(\divg0,\curl,\O)=\{\u,\;\curl\u\in L^2(\O,\Bbb R^3):\divg\u=0\text{ in $\O$,}\; \u\times\nu=\0 \text{ on $\p\O$}\},\\
&H_{n0}(\divg,\curl,\O)=\{\u,\;\curl\u\in L^2(\O,\Bbb R^3):\divg\u\in L^2(\O),\; \u\cdot\nu=0 \text{ on $\p\O$}\},\\
&H_{t0}(\divg,\curl,\O)=\{\u,\;\curl\u\in L^2(\O,\Bbb R^3):\divg\u\in L^2(\O),\; \u\times\nu=\0 \text{ on $\p\O$}\}.\\
\endaligned
$$

Before stating the main results, we review some recent researches. For three-dimensional bounded domains, it holds that $\alpha_1=\beta_1$, see \cite{Pauly}. Moreover, if the domain is convex, Pauly \cite{Pauly2015} proved that $\alpha_1=\beta_1\geq \mu_2$. For three-dimensional bounded and star-shaped $C^{1,1}$ domains, Zeng and the author \cite{ZZ2017} obtained that $\alpha_1=\beta_1<\gamma_1$. For two-dimensional bounded domains, Kelliher \cite{Kelliher} showed that $\gamma_k>\lambda_k$ holds for any positive integer $k$. For two-dimensional simply connected bounded domains, since the Stokes eigenvalue problem can be rewritten as the clamped buckling plate problem, one can obtain that $\gamma_1>\lambda_2$, see \cite{Payne1955} or \cite{Fri2004}. By the way, for the lower bound on the eigenvalues of the Stokes operator we refer Berezin-Li-Yau type inequalities, which were studied by \cite{Ilyin2010,YY2012}.

In this paper, we mainly obtain several comparison results for eigenvalues of the $\curl\curl$ operator and the Stokes operator. Firstly, for three-dimensional simply connected bounded domains, we prove that $\alpha_k=\beta_k<\gamma_k$, which is proved via an adaptation of Filonov's elegant proof \cite{Filonov2004} of the inequality $\mu_{k+1}<\lambda_k$. Secondly, using the fact that the first eigenvalue of the Dirichlet Laplacian is simple, we get that the first eigenvalue of the Stokes operator is strictly larger than the first eigenvalue of the Dirichlet Laplacian for any dimension $n\geq2$. Lastly, we obtain that $\alpha_1=\beta_1>\mu_2$ for strictly convex $C^{1,1}$ domains.

Now we state our main results more precisely.
\begin{Thm}\label{M-S}
Let~$\O$~be a simply connected domain. Then it holds that $\alpha_k=\beta_k<\gamma_k$, where $k$ is any positive integer.
\end{Thm}
\begin{Thm}\label{Thm2}
Let~$\O$ be a bounded $C^1$ domain in $\mathbb{R}^n$, where $n\geq2$. Then it holds that $\gamma_1>\lambda_1$.
\end{Thm}
\begin{Thm}\label{Thm3}
Let~$\O$~be a strictly convex $C^{1,1}$ domain. Then it holds that $\alpha_1=\beta_1>\mu_2$.
\end{Thm}
\begin{Rem}\label{Rem}
Using Theorem \ref{Thm3}, we can give a partial answer to a conjecture proposed by Pauly. Assume that $\O$ is convex. The Maxwell constants $c_{m,t}$ and $c_{m,n}$ are the best constants for the Maxwell inequalities
$$\|\u\|_{L^2(\O)}\leq c_{m,t}(\|\divg\u\|_{L^2(\O)}^2+\|\curl\u\|_{L^2(\O)}^2)^\frac{1}{2},\text{ for any $\u\in H^1_{t0}(\divg,\curl,\O)$,}$$
$$\|\u\|_{L^2(\O)}\leq c_{m,n}(\|\divg\u\|_{L^2(\O)}^2+\|\curl\u\|_{L^2(\O)}^2)^\frac{1}{2},\text{ for any $\u\in H^1_{n0}(\divg,\curl,\O)$,}$$
respectively. Pauly \cite[Theorem 5]{Pauly2015} proved that
$$\sqrt{\frac{1}{\lambda_1}}\leq c_{m,t}\leq c_{m,n}=\sqrt{\frac{1}{\mu_2}}.$$
Pauly \cite[Remark 11]{Pauly20151} conjectured that
$$\sqrt{\frac{1}{\lambda_1}}<c_{m,t}<c_{m,n}=\sqrt{\frac{1}{\mu_2}}.$$
For strictly convex $C^{1,1}$ domains, we show that $c_{m,t}<c_{m,n}$. For details, see Section 2.
\end{Rem}

\section{Proof of the main results}
For convenience, denote by $\mathrm{M}_{t0}$ the $\curl\curl$ operator on $H_{t0}(\divg0,\curl,\O)\cap \Bbb H_2(\O)^\perp$ and by $\mathrm{M}_{n0}$ the $\curl\curl$ operator on $H_{n0}(\divg0,\curl,\O)\cap \Bbb H_1(\O)^\perp$ .
\begin{Lem}\label{lem-sum}
For all $\mu$ we have
$$H^1_0(\divg0,\O)\cap\mathrm{ker}(\mathrm{M}_{n0}-\mu)=\{\0\}.$$
\end{Lem}
\begin{proof}
Let $\mathbf{u}\in H^1_0(\divg0,\O)\cap\mathrm{ker}(\mathrm{M}_{n0}-\mu)$. Set
$$\w(x)=
\begin{cases}
\mathbf{u}(x) &\text{if $x\in \O$},\\
\0 &\text{if $x\notin \O$}.
\end{cases}
$$
Then $\mathbf{w}\in H^1_0(\divg0,\mathbb{R}^3)$. For any $\mathbf{v}\in C^\infty_c(\mathbb{R}^3,\mathbb{R}^3)$, let $\varphi\in H^1(\O)\cap \mathbb{R}^\perp$ solve the following equation
\begin{equation*}
\begin {cases}
\Delta\varphi=\divg\mathbf{v} & \text{ \rm in }\O,\\
\frac{\p\varphi}{\p \nu}=\mathbf{v}\cdot\nu &\text{ \rm on }  \p\O.\\
\end {cases}
\end{equation*}
Then $\mathbf{v}-\nabla\varphi\in H_{n0}(\divg0,\curl,\O)$. Thus we have
$$
\aligned
\int_{\mathbb{R}^3}\curl\w\cdot\curl\mathbf{v}\;dx&=\int_\O\curl\u\cdot\curl\mathbf{v}\;dx=\int_\O\curl\u\cdot\curl(\mathbf{v}-\nabla\varphi)\;dx\\
&=\mu\int_\O\u\cdot(\mathbf{v}-\nabla\varphi)\;dx=\mu\int_\O\u\cdot\mathbf{v}\;dx\\
&=\mu\int_{\mathbb{R}^3}\w\cdot\mathbf{v}\;dx,
\endaligned
$$
which implies $\curl\curl\w=\mu\w$. Using the identity
$$\curl\curl\w=-\Delta\w+\nabla(\divg\w)=-\Delta\w,$$
we obtain $-\Delta\w=\mu\w$. Consequently, $\w=\0$.
\end{proof}

\begin{proof}[Proof of Theorem \ref{M-S}]
For any $\mu>0$, it is not difficult to verify that
$$\dim \mathrm{Ker}(\mathrm{M}_{t0}-\mu )=\dim \mathrm{Ker}(\mathrm{M}_{n0}-\mu ).$$
By some functional analysis (see \cite[p. 438]{Pauly}), we know
$$\sigma(\mathrm{M}_{t0})=\sigma(\mathrm{M}_{n0}).$$
Hence $\alpha_k=\beta_k$, which can also be viewed as an application of \cite[Corollary 32]{Pauly-p}. So we only need to prove $\beta_k<\gamma_k$. Since $\O$ is simply connected, $\mathbb{H}_1(\O)=\{\0\}$. Hence we get
$$H_{n0}(\divg0,\curl,\O)\cap\mathbb{H}_1(\O)^\perp=H_{n0}(\divg0,\curl,\O).$$
We denote the counting functions of the Stokes operator $\mathrm{S}$ and $\mathrm{M}_{n0}$ by $N_\mathrm{S}$ and $N_{\mathrm{M}_{n0}}$:
$$N_\mathrm{S}(\mu)=\mathrm{card}(\sigma(S)\cap[0,\mu]),\;N_{\mathrm{M}_{n0}}(\mu)=\mathrm{card}(\sigma(\mathrm{M}_{n0})\cap[0,\mu]).$$
It is not difficult to verify that
$$
\aligned
&N_\mathrm{S}(\mu)=\max\{\dim L:\;L\subseteq H^1_0(\divg0,\O),\;\int_\O|\nabla\u|^2\;dx\leq \mu\int_\O|\u|^2\;dx,\;\u\in L\},\\
&N_{\mathrm{M}_{n0}}(\mu)=\max\{\dim L:\;L\subseteq H_{n0}(\divg0,\curl,\O),\;\int_\O|\curl\u|^2\;dx\leq \mu\int_\O|\u|^2\;dx,\;\u\in L\}.
\endaligned
$$
Let $\mu>0$. We choose a subspace $F$ of $H^1_0(\divg0,\O)$ such that $\dim F=N_\mathrm{S}(\mu)$ and
$$\int_\O|\nabla\u|^2\;dx\leq \mu\int_\O|\u|^2\;dx,\;\u\in F.$$
From Lemma \ref{lem-sum}, we know that the sum $F+\mathrm{ker}(\mathrm{M}_{n0}-\mu)$ is direct.

Let $\u\in F$ and $\mathbf{v}\in\mathrm{ker}(\mathrm{M}_{n0}-\mu)$. Then it holds that
$$
\aligned
\int_\O|\curl(\u+\mathbf{v})|^2\;dx&= \int_\O(|\curl\u|^2+2\curl\u\cdot\curl\mathbf{v}+|\curl\mathbf{v}|^2)\;dx\\
&=\int_\O|\nabla\u|^2\;dx+2\mu\int_\O\u\cdot\mathbf{v}\;dx+\int_\O|\curl\mathbf{v}|^2\;dx\\
&\leq \mu\int_\O|\u|^2\;dx+2\mu\int_\O\u\cdot\mathbf{v}\;dx+\mu\int_\O|\mathbf{v}|^2\;dx\\
&=\mu\int_\O|\u+\mathbf{v}|^2\;dx.
\endaligned
$$
Thus, we get
$$N_{\mathrm{M}_{n0}}(\mu)\geq \dim (F+\mathrm{ker}(\mathrm{M}_{n0}-\mu))=N_\mathrm{S}(\mu)+\dim \mathrm{ker}(\mathrm{M}_{n0}-\mu).$$
Set $\mu=\gamma_k$, then we have
$$\mathrm{card}(\sigma(\mathrm{M}_{n0})\cap[0,\mu))=N_{\mathrm{M}_{n0}}(\mu)-\dim \mathrm{ker}(\mathrm{M}_{n0}-\mu)\geq N_\mathrm{S}(\mu)\geq k,$$
which implies $\beta_k<\gamma_k$.
\end{proof}
\begin{Rem}
If $\O$ is not simply connected, we can not verify that
$$H^1_0(\divg0,\O)\subseteq H_{n0}(\divg0,\curl,\O)\cap\mathbb{H}_1(\O)^\perp.$$
So the idea in the proof of Theorem \ref{M-S} may not work.
\end{Rem}

\begin{proof}[Proof of Theorem \ref{Thm2}]
It is trivial that $\gamma_1\geq\lambda_1$. In order to prove $\gamma_1>\lambda_1$, we only need to show $\gamma_1\neq\lambda_1$. We assume that $\gamma_1=\lambda_1$. Since $\gamma_1$ can be attained, there exists $\0\neq \u\in H^1_0(\O,\mathbb{R}^n)$ with $\divg\u=0$ in $\O$ such that
\begin{equation}\label{eq-lam1}
\int_\O|\nabla\u|^2\;dx=\gamma_1\int_\O|\u|^2\;dx=\lambda_1\int_\O|\u|^2\;dx.
\end{equation}
On the other hand, for any $1\leq k\leq n$, in view of $u_k\in H^1_0(\O)$, it holds that
\begin{equation}\label{eq-lam2}
\int_\O|\nabla u_k|^2\;dx\geq\lambda_1\int_\O|u_k|^2\;dx.
\end{equation}
Consequently, \eqref{eq-lam1} and \eqref{eq-lam2} force that
$$\int_\O|\nabla u_k|^2\;dx=\lambda_1\int_\O|u_k|^2\;dx,\; 1\leq k\leq n.$$
Since the first eigenvalue $\lambda_1$ of the Dirichlet Laplacian is simple, we obtain $u_k=c_k\varphi_1$, where $c_k$ is a constant and $\varphi_1$ is one of the first eigenfunctions of the Dirichlet Laplacian. Using the divergence-free condition, we have
\begin{equation}\label{div0}
\divg\u=\sum_{k=1}^nc_k\frac{\p\varphi_1}{\p x_k}=0.
\end{equation}
This implies
$\varphi_1=0$, which is a contradiction.

\end{proof}

In order to prove Theorem \ref{Thm3}, we need the following lemma, which can be found in \cite[Theorem 3.1.1.1]{Gri1985} or \cite[Lemma 2.11]{ABDG1998}.

\begin{Lem}\label{lem-id}
Let~$\O$~be a $C^{1,1}$ domain. If $\u\in H^1(\O,\mathbb{R}^3)$ with $\u\cdot\nu=0$ on $\p\O$, then we have
\begin{equation*}
\aligned
\int_\O|\nabla\u|^2\;dx=&\int_\O(|\divg\u|^2+|\curl\u|^2)\;dx-\int_{\p\O}\mathscr{B}(\u_T,\u_T)\;dS,
\endaligned
\end{equation*}
where $\mathscr{B}$ is the curvature tensor of the boundary and $\u_T=\u-(\u\cdot\nu)\nu$.
\end{Lem}

\begin{proof}[Proof of Theorem \ref{Thm3}]
Note that a bounded and convex domain must be simply connected.
Since $\beta_1$ can be attained, there exists $\0\neq \u\in H_{n0}(\divg0,\curl,\O)$ such that
\begin{equation}\label{eq-b1}
\int_\O|\curl\u|^2\;dx=\beta_1\int_\O|\u|^2\;dx.
\end{equation}
For any constant vector $\mathbf{a}\in\mathbb{R}^3$, we have
$$\int_\O\u\cdot\mathbf{a}\;dx=\int_\O\u\cdot\nabla(\mathbf{a}\cdot x)\;dx=\int_{\p\O}(\u\cdot\nu)(\mathbf{a}\cdot x)\;dS-\int_\O\divg\u(\mathbf{a}\cdot x)\;dx=0.$$
Consequently, each component of $\u$ has mean zero. Hence it holds that
\begin{equation}\label{eq-b2}
\int_\O|\nabla\u|^2\;dx\geq\mu_2\int_\O|\u|^2\;dx.
\end{equation}
By Lemma \ref{lem-id}, we obtain
\begin{equation}\label{eq-b3}
\aligned
\int_\O|\nabla\u|^2\;dx=&\int_\O|\curl\u|^2\;dx-\int_{\p\O}\mathscr{B}(\u_T,\u_T)\;dS.
\endaligned
\end{equation}
Combining \eqref{eq-b1}, \eqref{eq-b2} and \eqref{eq-b3}, we get
$$\mu_2\int_\O|\u|^2\;dx\leq \beta_1\int_\O|\u|^2\;dx-\int_{\p\O}\mathscr{B}(\u_T,\u_T)\;dS.$$
Since $\O$ is strictly convex, $\mathscr{B}$ is positive definite. If $\beta_1\leq\mu_2$, then the above inequality implies
\begin{equation*}
\int_{\p\O}\mathscr{B}(\u_T,\u_T)\;dS=0.
\end{equation*}
 Thus by the above equality we get $\u_T=\0$ on $\p\O$. Hence $\u=\0$ on $\p\O$. Moreover, we have
 $$\u\in H^1_0(\divg0,\O)\cap\mathrm{ker}(\mathrm{M}_{n0}-\beta_1).$$
 By Lemma \ref{lem-sum}, we have $\u=\0$, which is a contradiction. Therefore, $\alpha_1=\beta_1>\mu_2$.
\end{proof}

\begin{proof}[Proof of Remark \ref{Rem}]
Since
$$\|\u\|_{L^2(\O)}\leq \sqrt{\frac{1}{\beta_1}}\|\curl\u\|_{L^2(\O)},\text{ for any $\u\in H_{n0}(\divg0,\curl,\O)$,}$$
by the same method in the proof of \cite[Theorem 5]{Pauly2015}, we obtain
$$\|\u\|_{L^2(\O)}^2\leq \frac{1}{\lambda_1}\|\divg\u\|_{L^2(\O)}^2+\frac{1}{\beta_1}\|\curl\u\|_{L^2(\O)}^2,\text{ for any $\u\in H_{t0}(\divg,\curl,\O)$.}$$
Hence it follows that
$$c_{m,t}=\max\left\{\sqrt{\frac{1}{\lambda_1}},\sqrt{\frac{1}{\beta_1}}\right\}.$$
Thanks to the inequality $\mu_2<\lambda_1$ and Theorem \ref{Thm3}, we have
$$c_{m,t}<\sqrt{\frac{1}{\mu_2}}=c_{m,n}.$$
\end{proof}

\subsection*{Acknowledgements.}
The author is grateful to his supervisor, Prof. Xingbin Pan, for guidance and constant encouragement. The referee is thanked for valuable
comments and suggestions that helped to improve the paper. The author was supported by the Doctoral Scientific Research Foundation of Anhui University of Technology (DT17100057).

\end {document}